\documentclass[a4paper,11pt,twoside]{amsart}
\usepackage[latin1]{inputenc}
\usepackage[T1]{fontenc}

\usepackage{a4wide}

\usepackage{ifpdf}
\ifpdf\usepackage[pdftex]{hyperref}
\else\usepackage[hypertex]{hyperref}\fi

\usepackage{amsmath,amssymb,amsthm}

\theoremstyle{plain}
\newtheorem{thm}{Theorem}[section]
\newtheorem{prop}[thm]{Proposition}

\newtheorem{cor}[thm]{Corollary}

\theoremstyle{definition}
\newtheorem{defn}[thm]{Definition}

\theoremstyle{remark}
\newtheorem{rem}[thm]{Remark}

\theoremstyle{remark}

\DeclareMathOperator{\supess}{supess}
\DeclareMathOperator{\supp}{supp}

\let\dsp=\displaystyle 

\def\R{\mathbb R}

\def\N{\mathbb N}

\def\Q{\mathbb Q}

\begin{document}

\title[Distribution-energy inequalities and entropies]{Balanced\\
  distribution-energy inequalities \\ and related entropy bounds }

\author{Michel Rumin}
\address{Laboratoire de Mathématiques d'Orsay\\
  UMR 8628 \\ CNRS et Université Paris-Sud\\
  91405 Orsay\\ France}

\email{michel.rumin@math.u-psud.fr}

\date{\today}

\begin{abstract} Let $A$ be a self-adjoint operator acting over a
  space $X$ endowed with a partition.  We give lower bounds on the
  energy of a mixed state $\rho$ from its distribution in the
  partition and the spectral density of $A$. These bounds improve with
  the refinement of the partition, and generalize inequalities by
  Li-Yau and Lieb--Thirring for the Laplacian in $\R^n$. They imply an
  uncertainty principle, giving a lower bound on the sum of the
  spatial entropy of $\rho$, as seen from $X$, and some spectral
  entropy, with respect to its energy distribution. On $\R^n$, this
  yields lower bounds on the sum of the entropy of the densities of
  $\rho$ and its Fourier transform.  A general log-Sobolev inequality
  is also shown. It holds on mixed states, without Markovian or
  positivity assumption on $A$.
\end{abstract}



% \thanks{Author supported in part by the French ANR-06-BLAN60154-01
%   grant.}

\maketitle

% \setcounter{tocdepth}{1}
% \tableofcontents

\section{Introduction and main results}
\label{sec:introduction}

Let $(X, \mu)$ be a $\sigma$-finite measure space, $V$ a separable
Hilbert space and $A$ a self-adjoint operator acting on
$$
\mathcal{H} = L^2(X,V) = L^2(X,\mu) \otimes V\,.
$$ 
The inequalities we will consider concern \emph{mixed states}, that is
positive trace class operators on $\mathcal{H}$. From the
quantum-mechanical viewpoint, they are positive linear combination of
\emph{pure states}, which are the orthogonal projections on functions
in $\mathcal{H}$; see \cite[\S23]{von-Neumann} or
\cite{Wikipedia}. More precisely, as in \cite{Rumin10}, we are looking
for integral controls on the density of a state $\rho$ from its energy
given by the trace
$$
\mathcal{E}(\rho) = \tau (\rho A)\,.
$$ 

The density function of the state, or more generally of a bounded
positive operator $P$ on $\mathcal{H}$, is a notion that extends the
restriction to the diagonal of $X$ of the $V$-trace of the kernel of
$P$. It may be defined as follows (see e.g. \cite[\S1.2]{Rumin10}):
given a measurable set $\Omega \subset X$, the trace
\begin{equation}
  \label{eq:1}
  \nu_P(\Omega) = \tau (\chi_\Omega P \chi_\Omega) = \tau(P^{1/2}
  \chi_\Omega P^{1/2})  
\end{equation}
defines a measure on $X$. For any Hilbert basis $(e_i)$ of
$\mathcal{H}$, it holds that
\begin{equation}
  \label{eq:2}
  \nu_P(\Omega) = \int_\Omega D\nu_P(x) d\mu(x) \quad \mathrm{where} \quad 
  D\nu_P(x) = \sum_i \|(P^{1/2}e_i)(x)\|_V^2
\end{equation}
is called the \emph{density function} of $P$. For instance, in the
case of a pure state $P= \pi_f$ with $\|f\|_\mathcal{H} = 1$, one has
$D \nu_P (x) = \|f(x)\|_V^2$. Also, when $V$ is finite dimensional, as
for operators acting on scalar valued functions, it turns out that $D
\nu_P$ is bounded if and only if $P$ is ultracontractive from $L^1(X)
$ to $L^\infty(X)$ with
\begin{equation}
  \label{eq:3}
  \|P\|_{1,\infty} \leq D(P)= \supess D\nu_P(x) \leq (\dim V)
  \|P\|_{1,\infty} \,,
\end{equation}
see e.g. \cite[Prop.~1.4]{Rumin10}.

The inequalities studied here rely on the knowledge of the spectral
measure associated to $A$. It is defined as follows.
\begin{defn}
  Let $A$ be a self-adjoint operator on $\mathcal{H}$ and consider the
  spectral projections $\Pi_\lambda= \Pi_{]-\infty,\lambda[}(A)$. We
  define the \emph{spectral measure} of a measurable set $\Omega
  \subset X$ by
  \begin{equation}
    \label{eq:4}
    F_\Omega (\lambda) = \nu_{\Pi_\lambda}(\Omega) = \tau (\Pi_\lambda
    \chi_\Omega \Pi_\lambda) \,,
  \end{equation}
  and the \emph{spectral density function} by
  \begin{equation}
    \label{eq:5}
    F_x(\lambda) = D \nu_{\Pi_\lambda}(x)\,.
  \end{equation}
\end{defn}

These functions are positive increasing (in the large sense) and left
continuous. In the sequel, if $\varphi : \R \rightarrow \R^+$ is an
increasing function, and $y \geq0$, we will set
\begin{displaymath}
  \varphi^{-1}(y) = \sup \{x\in \R \mid \varphi(x) \leq y\} \in
  [-\infty, +\infty]\,.
\end{displaymath}
It is a pseudo-inverse of $\varphi$, and right continuous when finite.

\subsection{Energy of a confined state and spectral bounds}
\label{sec:energy-conf-state}
Our first purpose is to give an inequality between the trace of a
state supported in a domain $\Omega$ and its energy.

\begin{thm}
  \label{thm:CLR}
  Let $A$ be a self-adjoint operator acting on $\mathcal{H}=
  L^2(X,V)$, and let $\rho$ be a non-zero state (positive trace class
  operator) supported in a set $\Omega \subset X$. Suppose that
  \begin{displaymath}
    \mathcal{E}^+(\rho) = \tau(\rho^{1/2} \max(A,0) \rho^{1/2})\
    \mathrm{is\ finite.}
  \end{displaymath}
  Then the integral involved below has a finite positive part and it
  holds that
  \begin{equation}
    \label{eq:6}
    \|\rho\|_\infty \varphi_\Omega\Bigl(\frac{\tau(\rho)}{\|\rho\|_\infty} \Bigr) 
    \leq \mathcal{E}(\rho)\,,
  \end{equation}
  where $\dsp \varphi_\Omega(y) = \int_0^y F_\Omega^{-1}(u) du$ and
  $\|\rho\|_\infty$ denotes the $L^2-L^2$ norm of $\rho$.
\end{thm}

This sharpens and extends Theorem~1.7 in \cite{Rumin10}, restricted
there to positive operators.  When applied to projections onto
$N$-dimensional spaces $L$ of functions supported in $\Omega$,
Theorem~\ref{thm:CLR} gives a lower bound on the sum of the $N$-first
Dirichlet eigenvalues of $\mathcal{E}$ in $\Omega$, namely
\begin{equation}
  \label{eq:7}
  \varphi_\Omega(N) \leq \sum_{k=1}^N \lambda_k(\Omega) \leq
  \mathcal{E}(\Pi_L)\,. 
\end{equation}
Here the Dirichlet spectrum is defined using the min-max principle
\begin{displaymath}
  \lambda_n(\Omega) = \inf_{L \in \mathcal{L}_n }\max_{f
    \in L } (\mathcal{E}(f)/ \|f\|_2^2 ) \quad \mathrm{with}\quad
  \mathcal{L}_n = \{\supp L 
  \subset \Omega \mid \dim L = n\} \,.
\end{displaymath}

As we shall see in \S\ref{sec:illustration-rn}, this leads in the case
of the Laplacian in $\R^n$ to inequalities due to Berezin and Li-Yau
(\cite{Li-Yau} or \cite[Thm.~12.3]{Lieb-Loss}); and which are known to
be sharp in the semiclassical limit, i.e. when $N$ goes to $\infty$.
Also \eqref{eq:7} provides the following controls of the whole
Dirichlet spectral distribution in $\Omega$.

\begin{cor}
  \label{cor:CLR}
  Let $A$ and $\Omega$ as above, and let
  \begin{displaymath}
    N_\Omega(\lambda) = \sup \{\dim V \mid \mathrm{supp}\, V
    \subset \Omega \ \mathrm{and}\ \mathcal{E} \leq \lambda \
    \mathrm{on}\ V\}
  \end{displaymath}
  denotes the Dirichlet spectral distribution function of $A$ in
  $\Omega$. Then one has
  \begin{equation}
    \label{eq:8}
    \varphi_\Omega(N_\Omega(\lambda))  \leq \lambda N_\Omega(\lambda)
    \,.
  \end{equation}
  If moreover $A$ is positive, then
  \begin{equation}
    \label{eq:9}
    N_\Omega(\lambda) \leq 2 F_\Omega(2\lambda)\,.
  \end{equation}
\end{cor}

Hence in the positive case, the \emph{confined} spectral distribution
in $\Omega$ is controlled by twice the \emph{free} spectral density
seen from $\Omega$ at twice energy level, \emph{i.e.} by
$F_\Omega(2\lambda) = \tau (\chi_\Omega \Pi_{2\lambda} )$. Indeed
there, $\Pi_{2\lambda}$ is the free (or unconstrained) spectral space
of $A$ on the whole $X$.

% We shall see in \S?? that \eqref{eq:9} is sharp in this generality,
% in the sense that one can't improve both constants $2$ there.

\smallskip

One feature of the sharpness of inequalities like \eqref{eq:6} or
\eqref{eq:8}, that will be used in their proofs, lies in the fact they
stay equivalent under an \emph{energy shift} of $A$ in $A+ k$. Indeed,
one has then
\begin{displaymath}
  F_\Omega(\lambda) \rightarrow F_\Omega(\lambda-k) \quad
  \mathrm{thus} \quad 
  F_\Omega^{-1} \rightarrow F_\Omega^{-1} + k\quad \mathrm{and}\quad
  \varphi_\Omega(y) 
  \rightarrow \varphi_\Omega(y) + ky \,.
\end{displaymath}
Hence both sides of \eqref{eq:6} shift by $k \tau(\rho)$, while
\eqref{eq:8} stays unchanged up to a parameter shift. This implies in
particular that one can't improve \eqref{eq:6} or \eqref{eq:8} by a
fixed multiplicative factor for any (even positive) operator and
state. Indeed, suppose that for any positive operator and state it
holds
\begin{displaymath}
  (1+ \varepsilon) \|\rho\|_\infty \varphi_\Omega (\tau(\rho)
  /\|\rho\|_\infty) \leq \mathcal{E}_A(\rho) \,. 
\end{displaymath}
Then one would get by a positive energy shift $A \rightarrow A+k$ that
\begin{displaymath}
  0\leq (1+\varepsilon) \|\rho\|_\infty \varphi_\Omega (\tau(\rho)
  /\|\rho\|_\infty) \leq \mathcal{E}_{A- k\varepsilon} (\rho) < 0
\end{displaymath}
for $k$ large enough. Of course another stronger inequality than
\eqref{eq:6} may hold however.

In the sequel, we shall say that an inequality is \emph{balanced} if,
like \eqref{eq:6} or \eqref{eq:8}, it stays equivalent through energy
shift. None of the inequalities given in \cite{Rumin10} is balanced;
that precludes them to hold for operators of indefinite sign.

\subsection{A balanced Lieb-Thirring inequality}
\label{sec:balanc-lieb-thirr}

We now state a version of \eqref{eq:6}, that gives lower bounds on
$\mathcal{E}(\rho)$ knowing the distribution of the state in a
partition of $X $ into $\sqcup_i \Omega_i$, i.e. given $\nu_\rho
(\Omega_i) = \tau(\chi_{\Omega_i} \rho \chi_{\Omega_i})$.

\begin{thm}
  \label{thm:Lieb-Thirring}
  Let $A$ be a self-adjoint operator on $\mathcal{H}= L^2(X,V)$, and
  $\rho$ a non-zero state such that $\mathcal{E}^+(\rho)$ is
  finite. Let $\Omega_i$ be a measurable partition of $X$.

  $\bullet$ Then the sums and integral involved below have a finite
  positive part, and it holds that
  \begin{equation}
    \label{eq:10}
    H_{\Omega_i}(\rho) = \|\rho\|_\infty
    \sum_{i}
    \psi_{\Omega_i} \Bigl(\frac{\nu_\rho(\Omega_i)}{\|\rho\|_\infty}
    \Bigr) \leq H(\rho)=  
    \| \rho\|_\infty \int_X
    \psi_x\Bigl( \frac{D
      \nu_\rho(x)}{\|\rho\|_\infty } \Bigr) d\mu(x) \leq 
    \mathcal{E}(\rho) \,,
  \end{equation}
  where
  \begin{equation}
    \label{eq:11}
    \psi_{\Omega_i} (y)= \int_0^1 \varphi_{\Omega_i,t}(y) dt \quad
    \mathrm{with} \quad \varphi_{\Omega_i,t} (y)= \int_0^y
    F_{\Omega}^{-1}(t^2 u) du \,,
  \end{equation}
  and similarly
  \begin{equation}
    \label{eq:12}
    \psi_x (y) = \int_0^1 \varphi_{x,t} (y) dt \quad
    \mathrm{with} \quad
    \varphi_{x,t} (y) = \int_0^y   F_x^{-1}(t^2 u) du\,.
  \end{equation}

  $\bullet$ Moreover if $\Omega'_i$ is a finer partition of $X$ than
  $\Omega_i$, then $ H_{\Omega_i}(\rho) \leq H_{\Omega'_i}(\rho)$.
\end{thm}

These balanced inequalities improve the unbalanced ones given in
\cite[Thm.~1.6-1.7]{Rumin10} for positive operators. They extend an
inequality due to Lieb and Thirring in the case of the Laplacian on
$\R^n$; see \cite{Lieb-Thirring}, \cite[Thm.~12.5]{Lieb-Loss} and
\S\ref{sec:illustration-rn-1}.

To clarify its relation with the previous result, we first remark that
since $F_\Omega^{-1}$ is increasing, one has
\begin{equation}
  \label{eq:13}
  \psi_\Omega \leq
  \varphi_\Omega = \varphi_{\Omega,1} \,.
\end{equation}
Hence if the state is confined in a single domain $\Omega$ of the
partition, the bound \eqref{eq:6} is stronger than $ H_\Omega(\rho)
\leq \mathcal{E}(\rho)$ in \eqref{eq:10}. Conversely, we will see in
\S\ref{sec:proof-hrho-leq} that if $A$ is \emph{positive}, one has
\begin{equation}
  \label{eq:14}
  \varphi_\Omega \bigl( \frac{x}{2} \bigr) \leq \psi_\Omega (x)\,,
\end{equation}
thus \eqref{eq:10} in the confined case actually gives
$\|\rho\|_\infty \varphi_\Omega\bigl(
\frac{\nu_\rho(\Omega)}{2\|\rho\|_\infty} \bigr) \leq
\mathcal{E}(\rho)$, close to \eqref{eq:6}, but weaker.

\smallskip

From the quantum-mechanical viewpoint, \eqref{eq:10} gives a lower
bound on the energy that \emph{had} a state $\rho$ \emph{before} the
measure of its distribution in the partition, given by
$\nu_\rho(\Omega_i)= \tau(\chi_{\Omega_i} \rho \chi_{\Omega_i})$, is
performed. Equivalently, one gets an a priori control, through
$H_{\Omega_i}(\rho)$, on the possible outcomes of a measure of the
distribution of a state of known energy, before this measure is
done. Indeed, in quantum physics (see e.g.
\cite{von-Neumann,Wikipedia}), an actual measure of this distribution
collapses $\rho$ into
$$
\widetilde \rho = \sum_i \chi_{\Omega_i} \rho \chi_{\Omega_i}\,,
$$ 
which is a sum of localized states $\rho_i$ in $\Omega_i$. By
\eqref{eq:6} and convexity of $\varphi_{\Omega_i}$, one has then
\begin{equation}
  \label{eq:15}
  \| \widetilde \rho\|_\infty \sum_i  \varphi_{\Omega_i}
  \Bigl(\frac{\nu_\rho(\Omega_i))}{\|\widetilde \rho\|_\infty} \Bigr)   
  \leq \sum_i \|\rho_i\|_\infty  \varphi_{\Omega_i}
  \Bigl(\frac{\nu_\rho(\Omega_i))}{\|\rho_i\|_\infty} \Bigr) \leq 
  \sum_i \mathcal{E}(\rho_i) = \mathcal{E}(\widetilde \rho) \,.
\end{equation}
This is stronger than \eqref{eq:10} by \eqref{eq:13}, but applies only
to collapsed states.

\smallskip

The monotonicity of $H_{\Omega_i}(\rho)$ in the partition makes it
behave like an \emph{information quantity} on the state. It increases
with a finer knowledge of the distribution of $\rho$, and is dominated
by the continuous integral $H(\rho)$ associated to the
``infinitesimal'' distribution of $\rho$. Actually these inequalities
imply other information-type inequalities like entropy bounds, as we
see now.

\subsection{Spatial versus spectral entropy and uncertainty principle}
\label{sec:spatial-vs-spectral}

One interesting feature of the Lieb--Thirring inequality \eqref{eq:10}
lies in its simple behaviour under the change of $A$ into $f(A)$ for
an increasing right continuous function $f$. Indeed, one has
$\Pi_{f(A)} (]-\infty, \lambda[) \subset \Pi_A
(]-\infty,f^{-1}(\lambda)[) $, and thus for the spectral measures
\begin{equation}
  \label{eq:16}
  F_{f(A),\Omega}(\lambda) \leq F_{A,\Omega} \circ f^{-1} (\lambda) \,,
\end{equation}
This allows to change the integrals $H(\rho)$ in \eqref{eq:10} into
many expressions, while using the corresponding energy
$\mathcal{E}_{f(A)}(\rho) = \tau(f(A) \rho)$.

An attractive choice is to use $\ln F_A^+(A)$, where $F^+_A(\lambda)$
is the right limit of
\begin{displaymath}
  F_A(\lambda) = \supess_x F_{A,x}(\lambda)\,.
\end{displaymath}
For this application, it is crucial that \eqref{eq:10} holds for
non-positive operator, since $\ln F_A^+(A)$ is not positive in
general. This leads to entropy bounds.

\begin{thm}
  \label{thm:entropies}
  Let $A$ be a self-adjoint operator and $\rho$ a state such that
  $\mathcal{E}^+_{\ln F_A^+(A)}(\rho)$ is finite. Then the integral
  $S_\mu(\rho)$ below has a finite negative part and it holds that
  \begin{equation}
    \label{eq:17}
    S_\lambda(\rho) + S_\mu(\rho) \geq 0 \,,
  \end{equation}
  where
  \begin{equation}
    \label{eq:18}
    S_\lambda(\rho) =  \mathcal{E}_{\ln F_A^+(A)}(\rho)  \quad
    \mathrm{and} \quad S_\mu (\rho) = - \int_X \ln \Bigl(\frac{D 
      \nu_\rho(x)}{\|\rho\|_\infty}\Bigr) d \nu_\rho(x) + 3 \tau(\rho)
    \,.
  \end{equation}
  % \begin{displaymath}
  %   S_\mu (\rho) = \int_X \ln \Bigl(\frac{D
  %     \nu_\rho(x)}{\|\rho\|_\infty}\Bigr) d \nu_\rho(x) - 3
  %   \tau(\rho) \leq S_\lambda (\rho) = \mathcal{E}_{\ln
  %     F_A^+(A)}(\rho) \,.
  % \end{displaymath}
\end{thm}

The quantity $S_\mu(\rho)$ is related to the ``spatial entropy'' of
the state $\rho$, as seen from the measure space $X$. Actually,
\begin{displaymath}
  - S_\mu(\rho) + \tau(\rho)  (\ln \|\rho\|_\infty + 3) = \int_X \ln
  \bigr(\frac{d \nu_\rho}{d\mu}  
  \bigl) d \nu_\rho  = D_{KL}(\nu_\rho || \mu)
\end{displaymath}
is the Kullback--Leibler divergence from $\nu_\rho$ to $\mu$, or
relative entropy of $\nu_\rho$ to $\mu$. On the other hand,
\begin{equation}
  \label{eq:19}
  S_\lambda(\rho)  = \tau(\ln F_A^+(A) \rho) = \int_\R \ln
  F_A^+(\lambda) d \tau(\rho 
  \Pi_\lambda)  \,,
\end{equation}
deals with the ``spectral entropy'' of $\rho$, as seen from
distribution within the spectrum of $A$. Indeed $\ln F_A^+(\lambda)$
is an analytical \emph{ersatz} for $\ln \dim E_\lambda$ with
$E_\lambda = E_{]-\infty, \lambda]}(A)$. Namely, for invariant
operators acting on groups, one has $F_A^+(\lambda)= \dim_\Gamma
(E_\lambda) = \tau_\Gamma(\Pi_{]-\infty, \lambda]}(A))$ with the
notion of von Neumann's $\Gamma$-dimension; see
e.g. \cite[\S1]{Rumin10}.

This spectral entropy and the inequality \eqref{eq:17} have a striking
property: they are invariant under the change of $A$ into $f(A)$ for
any increasing homeomorphism $f$ of $\R$. Indeed the operator
$F_A^+(A)$ stays unchanged under such transforms, since they give
equality in \eqref{eq:16}. Thus, the spectral entropy is not sensitive
to the actual energy levels; it depends only on the ordered set
$\{\Pi_\lambda\}$, not its parametrization.

The quantities $S_{\mu, \lambda}(\rho)$ measure the indeterminacy in
position and energy of the state. They decrease respectively when
$\rho$ is concentrated in a set of small measure, or in small
energies. Notice that in the general case, if $X$ is not discrete and
$\mu(X)$ infinite, neither $S_\mu (\rho)$ nor $S_\lambda(\rho)$ are
bounded from below. Still, the lower bound for their sum in
\eqref{eq:17} means that the state can't be arbitrarily localized both
in position and energy. This may be seen as a general statement of the
uncertainty principle from the entropy viewpoint.

\subsection{Fourier transform and entropy}
\label{sec:four-transf-entr}

As an illustration of the previous result, we consider a state $\rho$
on $X=\R^n$, together with its Fourier transform $\widehat \rho$,
defined by $\widehat \rho (\widehat f) = \widehat{\rho (f)}$. We shall
see, by optimizing the choice of $A$ in \eqref{eq:17}, the following
bound on the sum of the entropy of the density of $\rho$ and the
entropy of the distribution of its Fourier transform.

\begin{thm}
  \label{thm:Fourier-entropies}
  Let $\mathrm{vol}^*$ be the measure $d^* \xi= (2\pi)^{-n } d \xi$ on
  $\R^n$, and $\rho$ as above. Consider the distribution function of
  $\nu_{\widehat \rho}$ relatively to $d^*\xi$
  \begin{equation}
    \label{eq:20}
    F_{\widehat\rho} (y) = \mathrm{vol}^* (\{\xi \in \R^n \mid  \frac{d
      \nu_{\widehat \rho}}{d^* \xi}(\xi) \geq y\})\,.  
  \end{equation}
  Suppose that the positive part of
  \begin{displaymath}
    S_F(\widehat\rho ) =  \int_0^{+\infty}  
    \ln (F_{\widehat\rho} (y))  F_{\widehat\rho} (y) dy 
  \end{displaymath}
  is finite. Then the negative part of
  \begin{displaymath}
    S_x(\rho) = -\int_{\R^n} \ln \bigl( \frac{d\nu_\rho}{dx}\bigr )
    d \nu_\rho(x) 
  \end{displaymath}
  is finite and it holds that
  \begin{equation}
    \label{eq:21}
    S_x(\rho) + S_F(\widehat \rho) \geq
    -\tau(\rho) (2 + \ln \|\rho\|_\infty) \,.
  \end{equation}
\end{thm}

This gives an \emph{operator free} version of the classical
uncertainty principle stating that a function (pure state) can't be
both arbitrarily localized in position and momentum. As will be seen
in \S\ref{sec:proof-theorem-fourier}, the bound \eqref{eq:21} is
equivalent to the previous one \eqref{eq:17}, with $A= \Delta$, for
states with \emph{spherical} density, but improves it on anisotropic
ones.

\smallskip

Still, the inequality \eqref{eq:21} is not symmetric in the roles of
$\rho$ and its Fourier transform $\widehat \rho$; because two kinds of
entropies are used at the space and frequency sides. However it
implies the following symmetric inequality.
\begin{cor}
  \label{cor:Fourier-entropies}
  It holds that
  \begin{equation}
    \label{eq:22}
    S_x (\rho) + S_\xi(\widehat \rho) = -\int_{\R^n} \ln \Bigl(
    \frac{d \nu_\rho}{d x} \Bigr) d 
    \nu_\rho -\int_{\R^n} \ln \Bigl( \frac{d \nu_{\widehat
        \rho}}{d^* \xi} \Bigr) d 
    \nu_{\widehat \rho}  \geq -\tau(\rho) (\ln \tau(\rho) + \ln
    \|\rho\|_\infty ) \,, 
  \end{equation}
  provided the positive part of one of these integrals is finite.
\end{cor}

This inequality has a better general behaviour than \eqref{eq:21};
namely it is additive on tensor products of trace one
states. Moreover, at least on projections on finite dimensional
spaces, the lower bound is related to von Neumann's proper entropy of
$\rho$, defined by $S(\rho) = -\tau(\rho\ln \rho)$; as discussed in
\S\ref{sec:around-corollary}.

\subsection{Log-Sobolev inequalities}
\label{sec:Log-sobol-ineq}

In the general setting, we finally observe that
Theorem~\ref{thm:entropies} is also related to more classical
$\log$-Sobolev inequalities, as stated for instance in \cite{Davies}
or \cite{Lieb-Loss} for the Laplacian. Namely, applying Jensen
inequality on the spectral entropy in \eqref{eq:17} leads to the
following entropy-energy bound.
\begin{cor}
  \label{cor:entropy-energy}
  Let $A$ be a self-adjoint operator and $\rho$ a state such that
  $\mathcal{E}_A^+(\rho)$ is finite and $\tau(\rho)= 1$. Then it holds
  that
  \begin{equation}
    \label{eq:23}
    -S_\mu(\rho) \leq (\ln F_A)^c (\mathcal{E}(\rho)) \,,
  \end{equation}
  where $(\ln F_A)^c$ is the concave hull of $\ln F_A$.
\end{cor}

This improves and extends Theorem~1.9 in \cite{Rumin10}, proved there
for positive operators and with a larger energy term. The inequality
\eqref{eq:23} is balanced and even invariant under an affine rescaling
of energy $A \rightarrow k_1 A + k_2$ with $k_1>0$.  It is also
equivalent to the family of parametric $\log$-Sobolev inequalities
\begin{equation}
  \label{eq:24}
  - S_\mu (\rho) \leq m(t) \tau(\rho) + t
  \mathcal{E}(\rho) \,,
\end{equation}
where $m(t) = \inf_{\lambda \geq 0} (\ln F(\lambda) -t\lambda)$ is
minus the concave-Legendre transform of $\ln F$. Such inequalities
actually hold on mixed states, without Markovian or positivity
assumption on $A$.

\section{The confined states inequalities}
\label{sec:proof-conf-stat}

\subsection{Proof of Theorem~\ref{thm:CLR}}
\label{sec:proof-theorem}

We first show Theorem~\ref{thm:CLR} for positive operator, and use
after the invariance through energy shift to extend it in the general
case.

The proof in the positive case is actually an improvement of an
argument given in \cite[\S3.1]{Rumin10}. Let $\Pi_{\geq \lambda} =
\Pi_{[\lambda, +\infty[}(A) = \mathrm{Id} - \Pi_\lambda$. We observe
that
\begin{equation}
  \label{eq:25}
  \mathcal{E}(\rho) = \tau(\rho^{1/2} A \rho^{1/2}) = \int_0^{+\infty} \tau
  (\rho^{1/2}\Pi_{\geq \lambda} \rho^{1/2} ) d\lambda \,.
\end{equation}
Since $\supp \rho \subset \Omega$, one has $\rho \leq \|\rho\|_\infty
\chi_\Omega$. Hence, assuming by homogeneity in $\rho$ that
$\|\rho\|_\infty= 1$ in the sequel, it holds that
\begin{align*}
  \tau (\rho^{1/2}\Pi_{\geq \lambda} \rho^{1/2} ) & = \tau(\rho) -
  \tau(\rho^{1/2}\Pi_{\lambda} \rho^{1/2}) = \tau(\rho) -
  \tau(\Pi_\lambda \rho \Pi_\lambda) \\
  & \geq \tau(\rho) -
  \tau(\Pi_\lambda \chi_\Omega \Pi_\lambda) \\
  & = \tau(\rho) - F_\Omega(\lambda) \,.
\end{align*}
Using it in \eqref{eq:25} for $\lambda < F_\Omega^{-1}(\tau(\rho)) =
\sup \{u \mid F_\Omega(u) \leq \tau(\rho)\}$ yields
\begin{align}
  \mathcal{E}(\rho) & \geq \int_0^{F_\Omega^{-1}(\tau(\rho))} \tau
  (\rho^{1/2} \Pi_{\geq \lambda} \rho^{1/2} ) d\lambda \nonumber \\
  & \geq \int_0^{F_\Omega^{-1}(\tau(\rho))} \bigl( \tau(\rho) -
  F_\Omega(\lambda) \bigr) d\lambda \nonumber \\
  & = \int_0^{F_\Omega^{-1}(\tau(\rho))} \Bigl(
  \int_{F_\Omega(\lambda)}^{\tau(\rho)} du \Bigr) d\lambda =
  \iint_{\{0 \leq F_\Omega(\lambda) \leq u \leq \tau(\rho)\}} du
  d\lambda
  \label{eq:26}\\
  &= \int_0^{\tau(\rho)}
  \Bigl( \int_0^{F_\Omega^{-1}(u)} d\lambda \Bigr) du \nonumber \\
  & = \int_0^{\tau(\rho)} F_\Omega^{-1}(u) du = \varphi_\Omega
  (\tau(\rho)) \,, \nonumber
\end{align}
as needed.

For a general self-adjoint operator, we consider $A_k = A \Pi_{\geq
  k}$. By positivity of $A_k - k$ and the behaviour of \eqref{eq:6} in
such a shift, it holds for any $k$ that
\begin{equation}
  \label{eq:27}
  \varphi_{A_k,\Omega}(\tau(\rho))
  \leq \mathcal{E}_{A_k} (\rho)\,.
\end{equation}
In particular, for $k=0$, one has $\max (F_A^{-1},0) \leq
F_{A_0}^{-1}$ and thus
\begin{displaymath}
  \int_0^{\tau(\rho)} \max (F_A^{-1}(u),0) du \leq
  \varphi_{A_0,\Omega} (\tau(\rho) ) \leq 
  \mathcal{E}_{A_0} (\rho)= \mathcal{E}^+(\rho) < \infty
\end{displaymath}
by hypothesis. Hence the integral $\varphi_{A,\Omega} (\tau(\rho)) =
\int_0^{\tau(\rho)} F_{A,\Omega}^{-1}(u) du$ makes sense in $[-\infty,
+\infty[$. If $\varphi_{A,\Omega} (\tau(\rho)) = -\infty$ there is
nothing more to prove, and we assume henceforth that
$\varphi_{A,\Omega} (\tau(\rho))$ is finite. This implies that the
increasing function $F_{A,\Omega}^{-1}(u)$ is finite for $u <
\tau(\rho)$. In particular, one has necessarily $F_{A,\Omega}(k)$
finite for $k \ll 0$, and thus by dominated convergence
\begin{equation}
  \label{eq:28}
  F_{A,\Omega}(k) = \tau(\chi_\Omega \Pi_{]-\infty, k[}(A)
  \chi_\Omega ) \searrow 0 \quad \mathrm{when}\quad k
  \searrow -\infty\,.
\end{equation}
Since, for $k \leq \lambda$, one has $\Pi_\lambda(A_k)
=\Pi_{[k,\lambda[}(A) = \Pi_\lambda(A) - \Pi_k(A)$, it holds that
\begin{equation}
  \label{eq:29}
  F_{A_k,\Omega}(\lambda) = \max(F_{A,\Omega}(\lambda) -
  F_{A,\Omega} (k),0)\,.
\end{equation}
This leads to $F_{A_k,\Omega}^{-1}(u) = F_{A,\Omega}^{-1}(u+
F_{A,\Omega}(k))$, and finally
\begin{displaymath}
  \varphi_{A_k,\Omega} (y) = \int_0^y F_{A_k,\Omega}^{-1}(u) du =
  \int_0^y F_{A,\Omega}^{-1}(u + F_{A,\Omega}(k) ) du \,.
\end{displaymath}
Together with \eqref{eq:28} and \eqref{eq:27}, this shows that
$\varphi_{A_k,\Omega} (\tau(\rho)) \searrow \varphi_{A,\Omega}
(\tau(\rho))$ when $k \searrow -\infty$; by dominated convergence for
the positive part, and monotone convergence for the negative one. For
the same reasons, one has $\mathcal{E}_{A_k}(\rho) \searrow
\mathcal{E}_A(\rho)$ for $k \searrow -\infty$, giving the result by
\eqref{eq:27}.

\subsection{Illustration in $\R^n$}
\label{sec:illustration-rn}

As a first illustration of the previous result, we consider the case
of the Laplacian on $X= \R^n$. By group invariance, the density
$F_x(\lambda)$ is a constant given by the value at $0$ of the kernel
of $\Pi_\lambda$. To compute it, we remark that the spectral spaces
$E_\Delta(\lambda)$ are functions whose Fourier transforms are
supported in the ball $B_n(0,\lambda^{1/2})$. It follows easily (see
e.g. \cite[\S4.2]{Rumin10}) that
\begin{equation}
  \label{eq:30}
  F_x(\lambda) = \widehat{\chi}_{B_n(0,\lambda^{1/2})}(0) = C_n
  \lambda^{n/2} \quad \mathrm{with} \quad C_n = (2\pi)^{-n}
  \mathrm{vol}(B_n(0,1)) \,, 
  % = (4\pi)^{-n/2}\Gamma(\frac{n}{2}+1)\,,
\end{equation}
and thus
\begin{equation}
  \label{eq:31}
  \varphi_\Omega (y) = \int_0^y F_\Omega^{-1}(u) du = \frac{n}{n+2} (C_n 
  \mathrm{vol}(\Omega))^{-2/n} y^{1+2/n}\,.
\end{equation}

Applying Theorem~\ref{thm:CLR} to the orthogonal projection $\rho$ on
the first $N$ Dirichlet eigenfunctions of $\Delta$ in $\Omega$, yields
the following inequality, due to Berezin and Li-Yau (see
\cite{Li-Yau}, \cite[Thm.~12.3]{Lieb-Loss})
\begin{equation}
  \label{eq:32}
  \mathcal{E}(\rho) = \sum_{i=1}^N \lambda_i(\Omega) \geq  \frac{n}{n+2} (C_n 
  \mathrm{vol}(\Omega))^{-2/n} N^{1+2/n} \,.
\end{equation}
When $\Omega$ is a domain of finite boundary area, this bound is known
to be sharp, up to lower order term in $N$, in the semiclassical
limit, i.e. for $N$ goes to $\infty$; see
e.g. \cite[Thm.~12.11]{Lieb-Loss},

\subsection{Equality case and bathtub filling}
\label{sec:equal-case-bathtub}

The proof of Theorem~\ref{thm:CLR} above shows that
$\mathcal{E}(\rho)$ gets smaller and approaches the proposed lower
bound $\varphi_\Omega(\tau(\rho))$ when:
\begin{enumerate}
\item $\rho$ is the largest possible, i.e close to $\chi_\Omega$, on
  $\Pi_\lambda$ for $\lambda < \lambda_0 = F_\Omega^{-1}(\tau(\rho))$;
\item and $\rho$ is the smallest possible, i.e. close to $0$, on
  $\Pi_{> \lambda_0}$.
\end{enumerate}
That means that $\rho$ has to fill up, or saturate, as much as
possible the lower energy levels it can, under the constraint that
$\rho \subset \subset \Omega$ and until the volume $\tau(\rho)$ is
reached. This kind of idea, clear from the physical viewpoint, is
actually quite similar to the \emph{bathtub principle} used in the
proof of Li-Yau inequality \eqref{eq:32} given by Lieb and Loss in
\cite[Theorem~12.3]{Lieb-Loss}.

In general, one can't have equality in \eqref{eq:6} unless $\rho$ is
pinched between $\Pi_{\lambda_0} = \Pi_{]-\infty, \lambda_0[}(A)$ and
$\Pi_{\lambda_0^+}= \Pi_{]-\infty, \lambda_0]}(A)$ \emph{and}
supported in $\Omega$. Hence, if the spectral spaces of $A$ are not
confined in a proper subspace $\Omega$ of the ambient space $X$, the
only remaining possibility is to take $\Omega=X$ itself. This requires
of course that $\dim E_{\lambda_0^-}= \tau(\Pi_{\lambda_0^- }) \leq
\tau(\rho)$ be finite.

\subsection{Asymptotic sharpness and amenability}
\label{sec:asympt-sharpn-amen}

One can go beyond the previous equality case and describe situations
with $X$ infinite and where \eqref{eq:6} is asymptotically
sharp. Given $\lambda$ and $\Omega$, one considers the two states
\begin{displaymath}
  \rho_\Omega = \chi_\Omega \Pi_\lambda \chi_\Omega   \quad \mathrm{and} \quad 
  \widetilde \rho_\Omega =  \Pi_\lambda \chi_\Omega \Pi_\lambda \,. 
\end{displaymath}
Notice that $\rho_\Omega$ is confined in $\Omega$ while $\widetilde
\rho_\Omega$ is not.  Still, one has $\tau(\rho_\Omega) =
\tau(\widetilde \rho_\Omega) = F_\Omega(\lambda)$ and we claim that
\begin{equation}
  \label{eq:33}
  \varphi_\Omega(\tau(\rho_\Omega))= \varphi_\Omega(F_\Omega(\lambda)) =
  \int_{]-\infty,\lambda[} u 
  dF_\Omega(u) =  \mathcal{E}(\widetilde \rho_\Omega)\,, 
\end{equation}
if this converges. To see this we proceed as in \eqref{eq:26},
assuming first that $A$ is positive. One finds
\begin{align*}
  \varphi_\Omega(F_\Omega(\lambda)) & =
  \int_0^{F_\Omega^{-1}(F_\Omega(\lambda)) }
  (F_\Omega(\lambda) - F_\Omega(u) ) du \\
  &= \int_0^\lambda (F_\Omega(\lambda) - F_\Omega(u) ) du \,,
\end{align*}
since $F_\Omega(u) = F_\Omega(\lambda)$ for $0\leq \lambda\leq u \leq
F_\Omega^{-1}(F_\Omega(\lambda))$. Thus
\begin{align*}
  \varphi_\Omega(F_\Omega(\lambda)) & = \int_{0\leq u \leq v < \lambda
  }
  dF_\Omega(v) du \\
  & = \int_{[0, \lambda[} v d F_\Omega(v) \,.
\end{align*}
The general case follows by energy cut-off and shift as in
\S\ref{sec:proof-theorem}.

When $\Omega$ is large, $\|\widetilde \rho_\Omega\|_\infty$ is close
to $1$, and \eqref{eq:33} means that \eqref{eq:6} is sharp for these
states $\widetilde \rho_\Omega$.  However they are not confined in
$\Omega$. Still $\mathcal{E}(\rho_\Omega)$ may be compared to
$\mathcal{E}(\widetilde \rho_\Omega)$ in the following situation. If
$X$ is a \emph{discrete} metric space, and $A$ is a bounded
\emph{local} operator, i.e. $Af(x)$ depends only on the value of $f$
in the ball $B(x,r)$, then one has
\begin{align*}
  |\mathcal{E}(\rho_\Omega) - \mathcal{E}(\widetilde \rho_\Omega)| &=
  |\tau(A \chi_\Omega \Pi_\lambda \chi_\Omega) - \tau (A \Pi_\lambda
  \chi_\Omega \Pi_\lambda)| \\
  & = |\tau(\Pi_\lambda \chi_\Omega (A\chi_\Omega - \chi_\Omega A) )|
  \\
  & \leq 2 \|A\|_\infty |\partial_r \Omega |\,,
\end{align*}
where $|\partial_r \Omega |$ is the cardinal of $\partial_r \Omega =
\{x \in X \mid d(x, \Omega) \leq r$ and $d(x,\Omega^c) \leq r\}$. This
leads to the following asymptotic sharpness result for \eqref{eq:6}.

\begin{prop}
  \label{prop:amenable}
  Let $X=\Gamma$ be a discrete infinite amenable group, endowed with
  an invariant measure, and let $A$ be a local translation invariant
  symmetric operator on $X$. Suppose that $\Omega_n$ is a F\"olner
  sequence such that $|\partial_r \Omega_n | / |\Omega_n| \rightarrow
  0$ when $n \rightarrow +\infty$. Set $F=F_x$ and $\varphi=
  \varphi_x$ (constant in $x$). Then it holds that
  \begin{equation}
    \label{eq:34}
    \lim_{n \rightarrow +\infty} \mathcal{E}(\rho_{\Omega_n}) /
    |\Omega_n| = \lim_{n \rightarrow +\infty} \|\rho_{\Omega_n}\|_\infty 
    \varphi_{\Omega_n}\bigl(\frac{\tau(\rho_{\Omega_n})}{\|
      \rho_{\Omega_n} \|_\infty } \bigr) / |\Omega_n| =
    \varphi(F(\lambda)) \,.
  \end{equation}
\end{prop}

This may be seen as the counterpart in the discrete setting to the
semiclassical result recalled in \S\ref{sec:illustration-rn}; here the
sharpness of \eqref{eq:6} is achieved on large domains and fixed
energy, instead of the contrary in the semiclassical limit. This
statement applies for instance to the discrete Laplacians on
$\ell^2$-cochains over amenable coverings of finite simplicial
complex.

\subsection{Faber--Krahn inequality and the heat technique}
\label{sec:faber-krahn-ineq}

We can compare the lower bound on the Dirichlet spectrum, or
Faber-Krahn inequality, obtained in \eqref{eq:7}:
\begin{equation}
  \label{eq:35}
  \lambda_1(\Omega) \geq \varphi_\Omega(1)\,,
\end{equation}
to the one shown in \cite[Prop.~II.2]{Coulhon2} using a heat kernel
technique. Namely, it follows from the Nash inequality given there
that if $A$ is a positive operator, one has
\begin{equation}
  \label{eq:36}
  \lambda_1(\Omega) \geq \theta(\Omega)= \sup_{t > 0}
  \frac{1}{t} \ln \Bigl( 
  \frac{1}{L(t) \mu(\Omega) }\Bigr)\,, 
\end{equation}
where $L(t) = \|e^{-tA}\|_{1,\infty}$. This bound is actually weaker
than \eqref{eq:35}, at least on scalar operators. Indeed, by
\eqref{eq:3}, it holds that
\begin{align*}
  L(t) \mu(\Omega) & \geq \nu_{e^{-tA}}(\Omega) = \tau(\chi_\Omega
  e^{-tA} \chi_\Omega) \\
  & = \int_0^{+\infty} e^{-t\lambda} d F_\Omega(\lambda) \\
  & \geq \int_{[0,F_\Omega^{-1}(1)]} e^{-t\lambda} d \widetilde
  F_\Omega(\lambda)
\end{align*}
with $\widetilde F_\Omega (\lambda)= F_\Omega(\lambda)$ for $\lambda<
F_\Omega^{-1}(1)$ and $\widetilde F_\Omega(F_\Omega^{-1}(1)) =
1$. Notice that $0\leq d \widetilde F_\Omega \leq d F_\Omega$ since
$F_\Omega(F_\Omega^{-1}(1)) \leq 1 \leq F_\Omega(F_\Omega^{-1}(1)^+)$
by left continuity of $F_\Omega$.
% and smaller than $F_\Omega$ since $F_\Omega(F_\Omega^{-1}(1)) \geq
% 1$, by right continuity of $F_\Omega$ if $F_\Omega^{-1}(1)$ is
% finite and $F_\Omega(+\infty) = \tau(\chi_\Omega) \geq 1$.
Then by Jensen,
\begin{align*}
  -\ln \bigl(L(t) \mu(\Omega)\bigr) & \leq t
  \int_{[0,F_\Omega^{-1}(1)]} \lambda
  d\widetilde F_\Omega(\lambda) \\
  & = t \int_0^1 (1-F_\Omega(\lambda)) du = t \varphi_\Omega(1)\,,
\end{align*}
by \eqref{eq:26}.  This gives $\theta(\Omega) \leq \varphi_\Omega(1)$
as claimed.

\subsection{Proof of Corollary~\ref{cor:CLR}}
\label{sec:comments-corollary-CLR}

When $A$ is a positive operator, one has for $c \in [0,1]$,
\begin{align}
  \label{eq:37}
  \varphi_\Omega(y) & = \int_0^y F_\Omega^{-1}(u) du
  \geq \int_{cy}^y F_\Omega^{-1}(u) du \\
  & \geq (1-c) y F_\Omega^{-1}(c y)\,. \nonumber
\end{align}
Hence \eqref{eq:8}, that comes from \eqref{eq:7}, implies
\begin{displaymath}
  N_\Omega(\lambda) \leq \frac{1}{c} F_\Omega \bigl(
  \frac{\lambda}{1-c} \bigr) \,,
\end{displaymath}
giving \eqref{eq:9} in the case $c=1/2$. Unlike \eqref{eq:8} these
inequalities are not balanced. If $F_\Omega$ is a concave function,
one can sharpen \eqref{eq:37} into $N_\Omega(\lambda) \leq 2
F_\Omega(\lambda)$ by Jensen. When $F_\Omega(\lambda) / \lambda$ is
increasing, for instance when $F_\Omega$ is a convex function, one
sees easily that $N_\Omega(\lambda) \leq F_\Omega (2\lambda)$.

\section{The balanced Lieb--Thirring inequality}
\label{sec:lieb-thirr-ineq}

We now consider Theorem~\ref{thm:Lieb-Thirring} and begin with the
continuous case. The argument is an improvement of
\cite[\S3.2]{Rumin10}.

\subsection{Proof of $H(\rho) \leq \mathcal{E}(\rho)$}
\label{sec:proof-hrho-leq}

Let $\rho$ be a state, $\Omega$ any measurable set in $X$, and let
consider the splitting
\begin{displaymath}
  \rho^{1/2} \chi_\Omega = \rho^{1/2} \Pi_\lambda \chi_\Omega +
  \rho^{1/2} \Pi_{\geq \lambda} \chi_\Omega \,.
\end{displaymath}
Using Hilbert-Schmidt norm and assuming by homogeneity that
$\|\rho^{1/2}\|_\infty = \|\rho\|_\infty^{1/2}= 1$ yield
\begin{align*}
  \| \rho^{1/2} \chi_\Omega \|_{HS} & \leq \|\rho^{1/2} \Pi_\lambda
  \chi_\Omega\|_{HS} +
  \|\rho^{1/2} \Pi_{\geq \lambda} \chi_\Omega\|_{HS} \\
  & \leq \|\Pi_\lambda \chi_\Omega\|_{HS} + \|\rho^{1/2}
  \Pi_{\geq\lambda} \chi_\Omega\|_{HS} \,.
\end{align*}
Since $\|P\|_{HS}= \tau(P^* P )^{1/2} = \tau(P P^*)^{1/2}$, one finds
by \eqref{eq:1} that
\begin{equation}
  \label{eq:38}
  \nu_\rho(\Omega)^{1/2} \leq \nu_{\Pi_\lambda}(\Omega)^{1/2}
  + \nu_{\Pi_{\geq \lambda}\rho \Pi_{\geq \lambda}}(\Omega)^{1/2} \,.
\end{equation}

This implies a similar inequality almost everywhere at the local
level, i.e.
\begin{equation}
  \label{eq:39}
  D \nu_\rho(x)^{1/2} \leq F_x(\lambda)^{1/2} +
  D \nu_{\Pi_{\geq\lambda}\rho \Pi_{\geq \lambda}} (x)^{1/2} \,.
\end{equation}
Indeed, using \eqref{eq:38} on the sets
\begin{displaymath}
  \Omega_{a,b,c} = \{ x \in X \mid D \nu_\rho(x) \geq a^2 \ , \  
  F_x(\lambda) \leq b^2 \ \mathrm{and}\ D \nu_{
    \Pi_{>\lambda}\rho \Pi_{> \lambda}} (x) \leq c^2 \} 
\end{displaymath}
with $(a,b,c) \in D = \{a,b,c \in \Q^+ \mid a > b + c\}$, gives that
$\mu (\Omega_{a,b,c}) = 0$. Whence
\begin{displaymath}
  \{ x \in X \mid \mathrm{~\eqref{eq:39} \ fails}\}  = \bigcup_D \Omega_{a,b,c}
\end{displaymath}
is also negligible. The author is grateful to Guy David for suggesting
this level set argument.

\smallskip

We now suppose that $A$ is positive, and uses \eqref{eq:25},
\begin{align}
  \mathcal{E}(\rho) & = \int_0^{+\infty} \tau(\rho^{1/2}\Pi_{\geq
    \lambda} \rho^{1/2} ) d\lambda = \int_0^{+\infty} \tau(\Pi_{\geq
    \lambda} \rho
  \Pi_{\geq \lambda} ) d\lambda \nonumber \\
  & = \int_0^{+\infty} \nu_{\Pi_{\geq
      \lambda} \rho \Pi_{\geq \lambda}} (X) d\lambda \nonumber \\
  &= \int_{X \times \R^+ } D \nu_{\Pi_{\geq
      \lambda} \rho \Pi_{\geq \lambda}} (x) d\mu(x) d\lambda \nonumber \\
  & \geq \int_\Omega D \nu_{\Pi_{\geq \lambda} \rho \Pi_{\geq
      \lambda}} (x) d\mu(x) d\lambda\,,
  \label{eq:40}
\end{align}
where $\Omega = \{(x,\lambda) \in X \times \R^+ \mid F_x(\lambda) \leq
D\nu_\rho(x) \}$. Then, by \eqref{eq:39},
\begin{align*}
  \mathcal{E}(\rho) & \geq \int_\Omega \bigl( D\nu_\rho(x)^{1/2} -
  F_x(\lambda)^{1/2} \bigr)^2 d\mu(x) d\lambda  \\
  & = \int_X \psi_x (D \nu_\rho (x)) d\mu(x)
\end{align*}
with
\begin{equation}
  \label{eq:41}
  \psi_x (y)   = \int_0^{F_x^{-1}(y)} \bigl( y^{1/2} -
  F_x(\lambda)^{1/2} \bigr)^2 d\lambda \,.
\end{equation}

We shall compare this expression to $\varphi_x(y) = \int_0^y
F_x^{-1}(t) dt$. First, using
\begin{displaymath}
  \sqrt y - \sqrt u \geq \sqrt{ \frac{y}{2} - u} \quad \mathrm{for}
  \quad 0\leq u \leq \frac{y}{2} \,,
\end{displaymath}
and proceeding as in \eqref{eq:26}, one finds that
\begin{align*}
  \psi_x (y) & \geq \int_0^{F_x^{-1}(y/2)} \bigr( \frac{y}{2} -
  F_x(\lambda) \bigr) d\lambda  \\
  & = \int_0^{y/2} F_x^{-1}(t) dt = \varphi_x \bigl( \frac{y}{2}
  \bigr) \,,
\end{align*}
This shows the comparison \eqref{eq:14} claimed for positive
operators. For the general expression \eqref{eq:12}, one uses
\eqref{eq:41}
\begin{align*}
  \psi_x(y) & = \int_0^{F_x^{-1}(y)} \int_{F_x(\lambda)}^y
  \bigl(u^{1/2} -
  F_x(\lambda)^{1/2} \bigr) \frac{du}{\sqrt u} d\lambda \, \\
  & = \int_0^{F_x^{-1}(y)} \int_{F_x(\lambda)}^y
  \int_{F_x(\lambda)}^u \frac{dv du d\lambda}{2\sqrt{uv}} \\
  & = \int_{\{0 \leq F_x(\lambda) \leq v\leq u \leq y\}} \frac{dv du
    d\lambda}{2 \sqrt{uv}}\\
  & = \int_0^y \int_0^u \int_0^{F_x^{-1}(v)} d\lambda \frac{dv}{2
    \sqrt v} \frac{du}{\sqrt u} \\
  & = \int_0^y \int_0^u F_x^{-1}(v)\frac{dv}{2
    \sqrt v} \frac{du}{\sqrt u}   \\
  & = \int_0^y \int_0^1 F_x^{-1}(t^2u) dt du \\
  & = \int_0^1 \varphi_{x,t}(y) dt \,,
\end{align*}
with $\varphi_{x,t}(y) = \int_0^y F^{-1}_x(t^2 u) du$ as needed. This
shows that $H(\rho) \leq \mathcal{E}(\rho)$ for positive operators.

 \begin{rem}
   The inequality $H_{\Omega_i}(\rho) \leq \mathcal{E}(\rho)$ for
   partitions can be proved along the same lines; just replacing
   \eqref{eq:40} above by its discrete analogous
   \begin{displaymath}
     \mathcal{E}(\rho) \geq \sum_i\int_0^{F_{\Omega_i}^{-1}(\nu_\rho(\Omega_i))}
     \nu_{\Pi_\lambda \rho 
       \Pi_\lambda} (\Omega_i) d\lambda \,,
   \end{displaymath}
   and using \eqref{eq:38} in place of \eqref{eq:39}. Furthermore, the
   previous computations on $\varphi_x$ and $\psi_x$ apply on
   $\varphi_{\Omega_i}$ and $\psi_{\Omega_i}$ instead.
 \end{rem}

 The case of general (non-positive) operators can be handled as in
 \S\ref{sec:proof-theorem}; using the cut-off $A_k = \max (A,k)$ and
 energy shift in these balanced inequalities. From the positive case,
 one has
 \begin{displaymath}
   \int_X \int_0^{D \nu_\rho(x)} \int_0^1 \max (F_x^{-1}(t^2u),0) \, dt du
   d\mu(x) \leq   \mathcal{E}^+(\rho) < \infty \,,
 \end{displaymath}
 Hence $\mathcal{E}^+(\rho)$ controls the positive part of the
 integral $H(\rho)$.  Then taking $k \searrow -\infty$ yields the
 result: by dominated convergence for the positive part and monotone
 convergence for the negative one.

 \subsection{Behaviour of $H_{\Omega_i}$ under partition refinement }
 \label{sec:behaviour-h-under}

 We shall now prove that
 \begin{displaymath}
   H_{\Omega_i} \leq H_{\Omega'_i}\leq H
 \end{displaymath}
 if $\Omega_i'$ is a finer partition of $X$ than $\Omega_i$. This will
 actually follow by integration in $t\in ]0,1]$ of the parametric
 inequalities
 \begin{equation}
   \label{eq:42}
   H_{\Omega_i,t} \leq H_{\Omega_i',t} \leq H_t
 \end{equation}
 where
 \begin{displaymath}
   H_{\Omega_i,t}(\rho) = \|\rho\|_\infty  \sum_i \varphi_{\Omega_i,t} \Bigl(
   \frac{\nu_\rho(\Omega_i)}{\|\rho\|_\infty} \Bigr) \quad \mathrm{and
   } \quad H_t(\rho) = \|\rho\|_\infty  \int_X \varphi_{x,t} \Bigl(
   \frac{D\nu_\rho(x)}{\|\rho\|_\infty} \Bigr) d\mu(x)\,.
 \end{displaymath}

\begin{rem}
  When $t=1$, we have seen in \eqref{eq:15} that these expressions
  give energy lower bounds of collapsed states, and \eqref{eq:42}
  means they also behave like an information quantity; actually finer
  than the averaged $H$, but restricted to such states.
\end{rem}

We start with the discrete vs. continuous inequality, in the positive
case, i.e. $F_x(0) = 0$, and assume again that $\|\rho\|_\infty =
1$. Given $t>0$,
\begin{displaymath}
  \varphi_{x,t}(y) = \int_0^y F_x^{-1} (t^2 u) du \quad \mathrm{and}
  \quad \varphi_{\Omega_i,t}(y) = \int_0^y F_{\Omega_i}^{-1} (t^2 u) du
\end{displaymath}
are convex functions whose Legendre transforms are respectively
\begin{displaymath}
  \varphi_{x,t}^*(z) = \int_0^z F_x(v) \frac{dv}{t^2} \quad
  \mathrm{and} \quad \varphi_{\Omega_i,t}^*(z) = \int_0^z
  F_{\Omega_i} (v) \frac{dv}{t^2} \,.
\end{displaymath}
Young's inequality states that for any $y,z \geq 0$
\begin{displaymath}
  yz \leq \varphi_{x,t}(y) +  \varphi_{x,t}^*(z) \,.
  % \quad \mathrm{and}
  % \quad yz \leq \varphi_{\Omega_i,t}(y) + \varphi_{\Omega_,t}^*(z)
  % \,.
\end{displaymath}
Integrating it over $\Omega_i$ with $y = D\nu_\rho(x)$ yields
\begin{align*}
  z \nu_\rho(\Omega_i) &\leq \int_{\Omega_i} \varphi_{t,x}
  (D\nu_\rho(x)) d\mu(x) + \int_{\Omega_i} \int_0^z F_x(v)
  \frac{dv}{t^2} d \mu(x)
  \\
  & = \int_{\Omega_i} \varphi_{x,t} (D\nu_\rho(x)) d\mu(x) +
  \varphi^*_{\Omega_i,t} (z) \,,
\end{align*}
by Fubini and \eqref{eq:5}.  Then by Legendre duality, one has
\begin{equation}
  \label{eq:43}
  \varphi_{\Omega_i,t}(\nu_\rho(\Omega_i)) = \sup_{z\geq 0} \bigl( z
  \nu_\rho(\Omega_i) - \varphi^*_{\Omega_i,t}(z) \bigr) \leq
  \int_{\Omega_i} \varphi_{x,t} 
  (D\nu_\rho(x)) d\mu(x)\,. 
\end{equation}
This gives $H_{\Omega_i,t}(\rho) \leq H_t(\rho)$ by summation. The
discrete comparison $H_{\Omega_i} (\rho)\leq H_{\Omega'_i}(\rho)$
follows the same lines: just replacing the integration over $\Omega_i$
above by the discrete splitting of $\Omega_i$ into smaller
$\Omega'_j$.

\smallskip

We now consider the general (non-positive) situation. From
\S\ref{sec:proof-hrho-leq}, the positive parts of $H_t(\Omega)$ and
$H_{\Omega_i,t}(\rho)$ are finite if $\mathcal{E}^+(\rho)$
is. Moreover we shall assume that the negative part of
$H_{\Omega_i,t}(\rho)$ is finite, or \eqref{eq:42} is already
satisfied. This implies in particular that $F_{\Omega_i}^{-1}(u) >
-\infty$ for any $i$ and $u>0$, and thus the functions
$F_{\Omega_i}(\lambda) = \int_{\Omega_i} F_x(\lambda) d\mu(x) \searrow
0$ when $\lambda \searrow -\infty$. Whence, fixing an $i$, one has
a.e. in $\Omega_i$ that $F_x(\lambda) \searrow 0$ when $\lambda
\searrow -\infty$. We shall now apply \eqref{eq:43} to
\begin{displaymath}
  F_{k,x}(\lambda) = F_x(\lambda+k) - F_x(k) \quad \mathrm{and}
  \quad F_{k,\Omega_i}(\lambda) = F_{\Omega_i}(\lambda + k) -
  F_{\Omega_i}(k) \,.
\end{displaymath}
This gives
\begin{displaymath}
  F_{k,x}^{-1} (u) = F_x^{-1}(u +
  F_x(k))) - k \quad \mathrm{and} \quad
  F_{k,\Omega_i}^{-1} (u) = F_{\Omega_i}^{-1}(u +
  F_{\Omega_i}(k))) - k \,,
\end{displaymath}
and
\begin{displaymath}
  \int_0^{\nu_\rho(\Omega_i)} F_{\Omega_i}^{-1}(t^2 u+ F_{\Omega_i}^{-1}(k)) du \leq
  \int_{\Omega_i} \int_0^{D\nu_\rho(x)} F_x^{-1}(t^2 u + F_x(k)) du d\mu(x)\,,
\end{displaymath}
leading to the result for $k \searrow -\infty$.

\subsection{Illustration in $\R^n$}
\label{sec:illustration-rn-1}

We consider again the case of the Laplacian on $\R^n$. From
\eqref{eq:30}, one has
\begin{displaymath}
  F_n^{-1}(u) = C_n^{-2/n} u^{2/n} = 4\pi \Gamma( 1+ n/2)^{2/n} u^{2/n}\,,
\end{displaymath}
giving
\begin{displaymath}
  \psi_n(y)  = \int_0^1\int_0^y F_n^{-1}(t^2 u )dudt = D_n y^{1+2/n} \,,
\end{displaymath}
with
\begin{displaymath}
  D_n = \frac{4\pi}{(1+4/n)(1+2/n)} \Gamma( 1+ n/2)^{2/n} \,.
\end{displaymath}
Thus, if $\rho$ is a projection onto a $N$-dimensional space of
orthonormal basis $f_i$, \eqref{eq:10} reads
\begin{equation}
  \label{eq:44}
  D_n \int_{\R^n} \Bigl( \sum_{i=1}^N |f_i(x)|^2 \Bigr)^{1+2/n} dx
  \leq \sum_{i=1}^N \| \nabla f_i\|_2^2 \,.  
\end{equation}
Such lower bound of the kinetic energy is due to Lieb--Thirring, see
\cite[Thm.~12.5]{Lieb-Loss} or \cite{Lieb-Thirring}, and have
applications in quantum mechanics. Notice that similar bounds can also
be obtained from \eqref{eq:10} for $|\nabla| = \Delta^{1/2}$ or the
relativistic kinetic energy $P = (\Delta + m^2)^{1/2} -m$ (see
\cite{Lieb-Loss}); replacing $F_n^{-1}$ above by $F_{|\nabla|}^{-1} =
(F_n^{-1})^{1/2}$ or $F_{P}^{-1}(\lambda) = (F_n^{-1} + m^2)^{1/2}
-m$.

The constant $D_n$ given here is quite sharp for large $n$. Indeed, by
\cite[\S12.5]{Lieb-Loss}, the (unknown) best constant has to be
smaller than $B_n= (1+4/n) D_n$. This follows from the remark that
\begin{displaymath}
  \varphi_n(y) = \int_0^y F_n^{-1}(u) du =  B_n y^{1 +2/n} \,.
\end{displaymath}
Indeed by Jensen inequality (or \eqref{eq:42}) and Berezin-Li-Yau
inequality \eqref{eq:32} one has both
\begin{displaymath}
  \varphi_{n,\Omega}(N) = \mu(\Omega)^{-2/n} \varphi_n (N) \leq 
  B_n \int_\Omega \Bigl(  \sum_{i=1}^N 
  |f_i(x)|^2 \Bigr)^{1+2/n} dx 
  \ \mathrm{and}\ \sum_{i=1}^N \| \nabla f_i\|_2^2 \,,
\end{displaymath}
for functions confined in a domain $\Omega$. As the second inequality
is sharp in the semiclassical limit $N\rightarrow +\infty$, the best
constant in \eqref{eq:44} is smaller than $B_n$ as claimed.

\section{Entropy bounds}
\label{sec:entropy-bounds}

\subsection{Proof of Theorem~\ref{thm:entropies}}
\label{sec:proof-theor-refthm}

We deduce the inequality between the spatial and spectral entropies
from Theorem~\ref{thm:entropies}.  Consider the functions
\begin{displaymath}
  F_A(\lambda)= \sup_{A,\Omega} \frac{F_{A,\Omega}(\lambda)}{\mu(\Omega)} =
  \supess_x F_{A,x}(\lambda) \quad \mathrm{and} \quad
  F_A^+(\lambda) = 
  \lim_{\varepsilon \rightarrow 0^+} F_A(\lambda + \varepsilon) \,.
\end{displaymath}
We observe that $F_A$ is increasing and left continuous, since the
$F_{A,\Omega}$ are, while $F_A^+ $ is right continuous. We shall
assume that $F_A(\lambda)$ is finite for $\lambda \ll 0$, in order
that the hypothesis of Theorem~\ref{thm:entropies} hold for some
state. This implies in particular by dominated convergence that
$F_A(\lambda) \searrow 0$ when $\lambda \searrow -\infty$. Then by
\eqref{eq:16}, one has
\begin{displaymath}
  F_{\ln F_A^+(A)} (\lambda) \leq F_A \circ (F_A^+)^{-1} 
  (e^\lambda ) \leq F_A \circ F_A^{-1}(e^\lambda) \leq e^\lambda ,
\end{displaymath}
by left continuity of $F_A$. Hence by \eqref{eq:12}, it holds a.e. in
$X$ that
\begin{equation}
  \label{eq:45}
  \psi_{\ln F_A^+(A),x} (y) \geq \int_0^1\int_0^y \ln(t^2u) du dt = y
  \ln y -3y \,,
\end{equation}
leading to Theorem~\ref{thm:entropies}.

\subsection{Illustration on $\R^n$}
\label{sec:illustration-rn-2}

We make explicit Theorem~\ref{thm:entropies} in the case of the
Laplacian on $\R^n$. Given a state $\rho$, we can express its spectral
entropy $S_\lambda(\rho)$ using Fourier transform. Suppose that $\rho
= \sum_i p_i \Pi_{f_i}$ for orthonormal functions $f_i$. Its Fourier
transform $\widehat \rho$ acts on $L^2(\R^n_\xi)$ by $\widehat \rho
(\widehat f) = \widehat{\rho (f)}$; actually $\widehat \rho = \sum_i
p_i \Pi_{\widehat{ f_i}}$ using the Plancherel measure $d^* \xi
=(2\pi)^{-n} d \xi$. At the density level, this writes
\begin{equation}
  \label{eq:46}
  d \nu_\rho(x) = \sum_i p_i |f_i(x)|^2 dx \quad \mathrm{and} \quad
  d\nu_{\widehat \rho}( \xi) = \sum_i p_i |\widehat{f_i}(\xi)|^2 d^* \xi\,.  
\end{equation}

By \eqref{eq:30}, $F_n(\lambda) = C_n \lambda^{n/2}$ and the spectral
entropy is
\begin{align*}
  S_\lambda(\rho) & = \tau (\ln (F_n(\Delta)) \rho) = \sum_i p_i
  \langle \ln (C_n \Delta^{n/2}) f_i ,
  f_i \rangle \\
  & = \sum_i p_i \int_{\R^n} \ln (C_n \|\xi\|^n ) | \widehat{
    f_i}(\xi)| ^2 d^* \xi
  \\
  & = \int_{\R^n} \ln (\mathrm{vol}^* (B_n(0, \|\xi\|)) d\nu_{\widehat
    \rho}( \xi) \,.
\end{align*}
Hence the entropy bound \eqref{eq:17} writes here
\begin{equation}
  \label{eq:47}
  \int_{\R^n} \ln \bigl( \frac{d \nu_\rho}{dx} \bigr) d \nu_\rho(x)  \leq 
  \int_{\R^n} \ln (\mathrm{vol}^*  (B_n(0, \|\xi\|))
  d\nu_{\widehat \rho}( \xi) + \tau(\rho) (3+ \ln \|\rho\|_\infty) \,. 
\end{equation}

To study the general sharpness of this bound, we first observe it
implies a log-Sobolev inequality. Indeed, Jensen inequality yields
\begin{align*}
  \int_{\R^n} \ln \bigl( \frac{d \nu_\rho}{dx} \bigr) \frac{d
    \nu_\rho(x)}{ \tau(\rho)} & \leq \ln F_n \Bigl( \int_{\R^n}
  \|\xi\|^2 \frac{d\nu_{\widehat \rho}( \xi)}{\tau(\rho)} \Bigr) +
  3+ \ln \|\rho\|_\infty \\
  & = \ln F_n \Bigl( \frac{\mathcal{E}(\rho)}{\tau(\rho)} \Bigr) + 3+
  \ln\|\rho\|_\infty \,.
\end{align*}
This in turn implies a Berezin--Li--Yau type inequality for confined
states in finite measure sets $\Omega$. Namely the convexity of
$y\mapsto y \ln y$ leads to
\begin{equation}
  \label{eq:48}
  \tau(\rho) \leq \mu(\Omega) \|\rho\|_\infty e^3 F_n \Bigl(
  \frac{\mathcal{E}(\rho)}{\tau(\rho)} \Bigr) \,. 
\end{equation}
This may be compared to \eqref{eq:6} where $\varphi_\Omega (y) =
\frac{n}{n+2} y F_n^{-1}(\frac{y}{\mu(\Omega)})$ gives
\begin{displaymath}
  \tau(\rho) \leq \mu(\Omega) \|\rho\|_\infty F_n \Bigl( \frac{n+2}{n}
  \frac{\mathcal{E}(\rho)}{\tau(\rho)} \Bigr) \,.
\end{displaymath}
As recalled in \S\ref{sec:illustration-rn} (and also
\S\ref{sec:asympt-sharpn-amen} in a discrete setting) this inequality
is sharp for all $n$, in the semiclassical limit of large energy. It
is indeed sharper than \eqref{eq:48}, since
\begin{displaymath}
  F_n \bigl( \frac{n+2}{n}
  \lambda \bigr) = C_n \Bigl( \frac{n+2}{n}
  \lambda \Bigr)^{n/2} \leq e F_n(\lambda) \,.
\end{displaymath}
As a consequence, the inequality \eqref{eq:47} is sharp except
possibly for the constant $3$ there, which can't be taken smaller than
$1$ in this generality.

\begin{rem}
  We notice that \eqref{eq:48}, with $e$ instead of $e^3$, is also an
  instance of the general confined states result
  Theorem~\ref{thm:CLR}. Indeed when applied to $\ln F_A^+(A)$, one
  can use
  \begin{displaymath}
    \varphi_{\ln F_A^+(A), \Omega} (y) \geq \int_0^y \ln \bigl(
    \frac{u}{\mu(\Omega)} \bigr) du = y \ln \bigl(
    \frac{y}{\mu(\Omega)} \bigr)
    - y \,,
  \end{displaymath}
  instead of the weaker (not confined) $\psi$ version \eqref{eq:45}.
\end{rem}

\subsection{Proof of Theorem~\ref{thm:Fourier-entropies}}
\label{sec:proof-theorem-fourier}

The right spectral term of the previous entropy bound \eqref{eq:47} is
associated to the level sets of the symbol $\sigma_\Delta(\xi) =
\|\xi\|^2$ of the Laplacian; namely at some point $\xi_0$, one has
$B_n (0, \| \xi_0\|)) = \{\xi \mid \sigma_\Delta(\xi) \leq
\sigma_\Delta(\xi_0)\}$, whose volume gives the spectral density
$F_\Delta(\lambda)$ at the energy $\lambda= \|\xi_0\|^2$. Given a
state $\rho$, one can consider more general translation invariant
operators, associated to other fillings of $\R^n_\xi$, in order to
minimize the spectral entropy term $S_\lambda(\rho)$. We shall proceed
as follows.

Let $\sigma$ be a measurable bounded function on $\R^n_\xi$, and
$A_\sigma$ be defined by
\begin{displaymath}
  \widehat{A_\sigma(f)} (\xi) = \sigma(\xi) \widehat f(\xi)\,.
\end{displaymath}
Let $\Omega^\sigma_\lambda= \{\xi \in \R^n \mid \sigma(\xi) \leq
\lambda\}$. The spectral projection $\Pi_{A_\sigma}(\lambda)$ acts
through Fourier transform by multiplication by $\chi_{\Omega_\lambda}$
and, following e.g. \cite[\S4.1, \S4.2]{Rumin10}, the spectral density
of $A_\sigma$ is
\begin{displaymath}
  F_{A_\sigma}^+(\lambda) = \|k_{\Pi_{A_\sigma}(\lambda)}\|_{L^2_x}^2 =
  \|\chi_{\Omega_\lambda^\sigma}\|_{L^2_\xi}^2 = \mathrm{vol}^* (
  \Omega_\lambda^\sigma) \,.
\end{displaymath}
This leads to the following expression for the spectral entropy of a
state $\rho$, as long these integral have finite positive parts,
\begin{align*}
  S_{A_\sigma}(\rho) & = \tau (\ln F_{A_\sigma}^+(A_\sigma)\rho) =
  \int_\R \ln (\mathrm{vol}^* ( \Omega^\sigma_\lambda)) d \tau
  (\Pi_{A_\sigma}(\lambda) \rho)
  \\
  & = \int_\R \ln (\mathrm{vol}^* ( \Omega_\lambda^\sigma)) d
  \tau(\chi_{\Omega_\lambda^\sigma} \widehat \rho)
  \\
  & = \int_\R \ln (\mathrm{vol}^* ( \Omega_\lambda^\sigma)) d
  \nu_{\widehat \rho}(\Omega_\lambda^\sigma)
  \\
  & = \int_\R \ln (\mathrm{vol}^* ( \Omega_\lambda^\sigma)) d
  (\sigma_*( \nu_{\widehat\rho})(]-\infty,\lambda]))\,,
\end{align*}
using the push-forward measure $\sigma_*(\nu_{\widehat \rho})$. This
yields
\begin{equation}
  \label{eq:49}
  S_{A_\sigma}(\rho)  = \int_{\R^n} \ln (\mathrm{vol}^* (
  \Omega_{\sigma(\xi)}^\sigma))
  d\nu_{\widehat\rho}(\xi) \,.
\end{equation}
The strong functional invariance of this entropy is clear here. It
stays unchanged if replacing the symbol $\sigma$ into $f(\sigma)$ for
any strictly increasing function $f$ on $\sigma(\R^n)$, as comes from
$\Omega^{f(\sigma)}_{f(\sigma)(\xi)} =
\Omega^\sigma_{\sigma(\xi)}$. The following statement gives the
minimum of these quantities and implies
Theorem~\ref{thm:Fourier-entropies}.

\begin{prop}
  \label{prop:bathtub}
  Given $\rho$, let $g = \frac{d \nu_{\widehat \rho}}{d^*\xi}$ and
  $F_{\widehat\rho} (y) =\mathrm{vol}^* (\{\xi \mid g(\xi) > y\}) $ as
  in \eqref{eq:20}. Then one has
  \begin{displaymath}
    S_{A_\sigma}(\rho) \geq S_F(\widehat \rho) -\tau(\rho) =\int_0^{+\infty}  
    F_{\widehat\rho} (y)  \ln 
    F_{\widehat\rho} (y) dy - \tau(\rho) \,. 
  \end{displaymath}
  Equality holds if $\sigma$ is a decreasing \emph{regular filling} of
  the level sets of $g$ in the following sense:
  \begin{itemize}
  \item for each $y$, there exists $\lambda$ such that
    \begin{displaymath}
      \{ \xi \mid
      g(\xi) > y \} \subset \Omega_\lambda^\sigma = \{ \xi \mid
      \sigma(\xi) \leq \lambda\}  
      \subset \{ \xi \mid
      g(\xi) \geq y \} ;
    \end{displaymath}
  \item for all $\lambda$, $\mathrm{vol}^* (\sigma^{-1}(\lambda)) =
    0$.
  \end{itemize}
\end{prop}

Equivalently, the level sets $\Omega_\lambda^\sigma$ of a regular
filling $\sigma$ are the sets $\{ \xi \mid g(\xi) \geq y \}$ (up to
zero measure) for regular values of $\rho$, i.e. when
$\mathrm{vol}^*(g^{-1}(y)) = 0$, while on $g^{-1}(y)$ for the
(discrete) non-regular values, they interpolate continuously in
measure between $\{ \xi \mid g(\xi) > y \}$ and $\{ \xi \mid g(\xi)
\geq y \}$. This can be achieved since the measure has no atom.

\smallskip

From Proposition~\ref{prop:bathtub}, we notice that the use of the
Laplacian is already optimal to minimize the spectral entropy of
states with spherical density; Theorems~\ref{thm:entropies}
and~\ref{thm:Fourier-entropies} are equivalent on such states. On
anisotropic states, one advantage of \eqref{eq:21} over \eqref{eq:47}
lies in its stronger invariance through general linear transforms
$f(x) \mapsto f(Ax)$ and $\rho \mapsto \rho_A = A \rho A^{-1}$.  In
such cases, one checks easily that
\begin{displaymath}
  D \nu_{\rho_A}(x) = |\det A | D\nu_{\rho}(Ax) \quad \mathrm{while}
  \quad D \nu_{\widehat{\rho_A }}(\xi) = |\det A|^{-1} D
  \nu_{\widehat\rho}(^tA^{-1}\xi) \,,
\end{displaymath}
giving that
\begin{displaymath}
  S_x(\rho_A) = S_x(\rho) - \tau(\rho) \ln |\det A| \quad\mathrm{
    while} \quad
  S_F(\widehat \rho_A) = S_F(\widehat \rho) + \tau(\rho) \ln |\det A|\,,
\end{displaymath}
% \begin{displaymath}
%   -\int_{\R^n} \ln (D \nu_{\rho_A}) d \nu_{\rho_A}  = -\int_{\R^n} \ln
%   (D \nu_\rho) d \nu_\rho  - \tau(\rho) \ln |\det A| \,,
% \end{displaymath}
% and
% \begin{displaymath}
%   \int_0^{+\infty}  F_{\widehat{\rho_A}} (y)  \ln F_{\widehat{\rho_A}}
%   (y) dy =  
%   \int_0^{+\infty}  
%   F_{\widehat\rho} (y)  \ln F_{\widehat\rho} (y) dy + \tau(\rho) \ln
%   |\det A| \,,
% \end{displaymath}
which keeps the entropy sum unchanged in \eqref{eq:47}.

\begin{proof}[Proof of Proposition~\ref{prop:bathtub}]
  Let $v(\xi) = \mathrm{vol}^*(\Omega_{\sigma(\xi)}^\sigma)$ and
  $F_\sigma(z) = v_*(\nu_{\widehat \rho})(]0, z])$. Then by
  \eqref{eq:49}
  \begin{equation}
    \label{eq:50}
    S_{A_\sigma}(\rho) = \int_0^{+\infty} \ln y d F_\sigma(y) = -\int_0^1
    F_\sigma(y) \frac{dy}{y} + \int_1^{+\infty} (\tau(\rho) -
    F_\sigma(y)) \frac{dy}{y} \,,
  \end{equation}
  by Fubini, since $ F_\sigma(+\infty)= \tau(\widehat \rho) =
  \tau(\rho)$. Hence $S_{A_{\sigma_1}}(\rho) \geq
  S_{A_{\sigma_2}}(\rho)$ if $F_{\sigma_1} \leq F_{\sigma_2}$, and we
  have to look for upper bounds for $F_\sigma$ to minimize
  $S_{A_\sigma}(\rho)$.

  By definition, one has
  \begin{displaymath}
    F_\sigma(y)  = \int_{D_y^\sigma} g(\xi) d^*\xi  \quad
    \mathrm{with} \quad D_y^\sigma = \{\xi
    \mid \mathrm{vol^*} (\Omega^\sigma_{\sigma(\xi)}) \leq y\} \,.
  \end{displaymath}
  Clearly one has $\{\xi \mid \sigma(\xi) < \lambda\} \subset
  D_y^\sigma \subset \Omega^\sigma_\lambda$ where $\lambda =
  \sup_{D_y^\sigma} \sigma$, and thus $\mathrm{vol}^*(D_y^\sigma) \leq
  y$ with equality if $\mathrm{vol}^*(\sigma^{-1}(\lambda)) =
  0$. Hence by the 'bathtub principle' (see
  \cite[Theorem~12.3]{Lieb-Loss}) one has
  \begin{equation}
    \label{eq:51}
    F_\sigma(y) \leq F(y) = \int_{\{g > F_{\widehat \rho}^{-1}(y)\}} g(\xi)
    d^*\xi + F_{\widehat \rho}^{-1}(y) 
    (y - F_{\widehat \rho}(F_{\widehat \rho}^{-1}(y)))\,,
  \end{equation}
  with $F_{\widehat \rho}^{-1}(y) = \inf \{z \mid F_{\widehat \rho}(z)
  \leq y\}$. Indeed, this comes from the identity
  \begin{multline*}
    F_\sigma(y) - F(y) = \int_{D_y^\sigma \cap \{g\leq F_{\widehat
        \rho}^{-1}(y)\}} (g (\xi)- F_{\widehat \rho}^{-1}(y)) d^*\xi
    \\ - \int_{(D_y^\sigma)^c \cap \{g > F_{\widehat \rho}^{-1}(y)\}}
    (g( \xi) - F_{\widehat \rho}^{-1}(y))d^*\xi + F_{\widehat
      \rho}^{-1}(y) (\mathrm{vol}^*(D^\sigma_y) -y) \,.
  \end{multline*}
  Moreover this shows that equality holds in \eqref{eq:51} if
  $\mathrm{vol}^*(D_y^\sigma) = y$ and, up to zero measure sets, $ \{g
  > F_{\widehat \rho}^{-1}(y) \} \subset D_y^\sigma \subset \{g \geq
  F_{\widehat \rho}^{-1}(y) \}$; which is fulfilled for regular
  fillings by the discussion above.
  
  We rewrite the function $F(y)$ in a more convenient form. Since
  $\tau(\rho) = \tau( \widehat \rho) = \int_{\R^n} g(\xi) d^*\xi$, one
  has
  \begin{align*}
    \int_{\{g > F_{\widehat \rho}^{-1}(y)\}} g(\xi) d^*\xi & =
    \tau(\rho) - \int_{\{g \leq F_{\widehat \rho}^{-1}(y)\}} g(\xi)
    d^*
    \xi \\
    & = \tau(\rho) - \int_{\{0\leq u < g(\xi) \leq F_{\widehat
        \rho}^{-1}(y) \}}
    du d^* \xi \\
    & = \tau (\rho) - \int_0^{F_{\widehat \rho}^{-1}(y)} (F_{\widehat
      \rho}(u) - F_{\widehat \rho}(F_{\widehat \rho}^{-1}(y)) du \\
    & = \tau (\rho) - \int_0^{F_{\widehat \rho}^{-1}(y)} F_{\widehat
      \rho}(u) du + F_{\widehat \rho}^{-1}(y) F_{\widehat
      \rho}(F_{\widehat \rho}^{-1}(y))\,.
  \end{align*}
  Then by \eqref{eq:51},
  \begin{align*}
    F(y) & = \tau(\rho) - \int_0^{F_{\widehat \rho}^{-1}(y)}
    (F_{\widehat
      \rho}(u) - y) du \\
    & = \tau(\rho) - \int_{\{ y < v < F_{\widehat \rho}(u) \}} dv du
    \,,
  \end{align*}
  since by right continuity of $F_{\widehat \rho}$, one has $u <
  F_{\widehat \rho}^{-1}(y)$ iff $F_{\widehat \rho} (u) >
  y$. Therefore
  \begin{equation}
    F(y) = \tau(\rho) - \int_y^{+\infty} F_{\widehat \rho}^{-1}(v)
    dv  = \int_0^y F_{\widehat \rho}^{-1}(v) dv 
    \label{eq:52}\,,
  \end{equation}
  since
  \begin{align*}
    \tau(\rho) & = \int_{\R^n} g(\xi) d^*\xi = \int_{\{0 \leq u <
      g(\xi)\}} du d^*\xi = \int_0^{+\infty} F_{\widehat \rho} (u) du
    \\
    & = \int_{\{ 0 \leq v < F_{\widehat \rho}(u)\}} dv du =
    \int_0^{+\infty} F_{\widehat \rho}^{-1}(v) dv\,.
  \end{align*}
  Finally, \eqref{eq:50} and \eqref{eq:52} lead to
  \begin{align*}
    S_{A_\sigma} (\rho) & \geq \int_0^{+\infty} \ln y F_{\widehat
      \rho}^{-1}(y)
    dy  = \int_{\{0 < z < F_{\widehat \rho}^{-1}(y) \}} \ln y dz dy \\
    & = \int_{\{0 < y < F_{\widehat \rho}(z)\} } \ln y dy dz \\
    & = \int_0^{+\infty} F_{\widehat \rho}(z) (\ln F_{\widehat
      \rho}(z) -1) dz = S_F(\widehat \rho) - \tau(\rho) \,,
  \end{align*}
  as claimed in Proposition~\ref{prop:bathtub}.
\end{proof}

\subsection{Proof of Corollary~\ref{cor:Fourier-entropies}}
\label{sec:comments-theorem}

We deduce Corollary~\ref{cor:Fourier-entropies} from
Theorem~\ref{thm:Fourier-entropies}.  This relies on the following
entropy comparison:
\begin{equation}
  \label{eq:53}
  S_F(\widehat \rho) \leq  S_\xi (\widehat \rho) +  \tau(\rho) (1 +
  \ln \tau(\rho))\,. 
\end{equation}
Indeed, one has
\begin{align*}
  -S_\xi(\widehat \rho) - \tau(\rho) & = \int_{\R^n} \ln \Bigl(
  \frac{d \nu_{\widehat \rho}}{d^* \xi} \Bigr) d \nu_{\widehat \rho}
  - \tau(\widehat \rho) \\
  & = \int_{\{0< y< \frac{d \nu_{\widehat \rho}}{d^*\xi}\}} \ln y dy
  d^* \xi \\
  & = \int_0^{+\infty} F_{\widehat \rho}(y) \ln y dy\,,
\end{align*}
thus
\begin{align*}
  S_F(\widehat \rho) - S_\xi(\widehat \rho) - \tau(\rho) & =
  \int_0^{+\infty} \ln (y
  F_{\widehat \rho}(y)) F_{\widehat \rho}(y) dy \\
  & \leq \int_0^{+\infty} \ln (\tau(\rho)) F_{\widehat \rho}(y) dy =
  \tau(\rho) \ln \tau(\rho)\,,
\end{align*}
since $y F_{\widehat \rho}(y) = y \mathrm{vol}^* \{ \xi \mid \frac{d
  \nu_{\widehat \rho}}{d^* \xi}(\xi) > y \} \leq \int_{\R^n} \frac{d
  \nu_{\widehat \rho}}{d^* \xi} d^*\xi = \tau (\rho)$. Then,
\eqref{eq:21} and \eqref{eq:53} give
\begin{displaymath}
  S_x(\rho) + S_\xi(\widehat \rho) \geq -\tau(\rho) (3 + \ln \tau(\rho) + \ln
  \|\rho\|_\infty ) \,.
\end{displaymath}

Then we observe that, except for the term $-3 \tau(\rho)$, this
expression is additive in taking tensor product of unit trace
states. Therefore, applying it to $\displaystyle \otimes^N \rho$ on
$\R^{nN}$, and dividing by $N$ for $N \nearrow +\infty$, gives
\eqref{eq:22} on unit trace states, and the general statement by
homogeneity.

\subsection{Around Corollary~\ref{cor:Fourier-entropies}}
\label{sec:around-corollary}

It is appealing trying to express, or bound, the right side of
\eqref{eq:22} using von Neumann's entropy of $\rho$. Following
\cite{von-Neumann,Wikipedia} this intrisic entropy is defined for unit
trace states by
\begin{displaymath}
  S(\rho) = - \tau(\rho \ln \rho) \,.
\end{displaymath}
For such states, one has $S(\rho) \geq - \ln \|\rho\|_\infty$, with
equality on normalized projections on finite dimensional spaces $\rho
= \Pi_V / \dim V$. Hence on these projections, \eqref{eq:22} reads
\begin{equation}
  \label{eq:54}
  S_x(\rho) + S_\xi(\widehat \rho) \geq  S(\rho) \, (= \ln \dim V ) .
\end{equation}

\smallskip

We don't know whether this holds for general unit trace states. An
interesting family of examples here is given by the heat of the
harmonic oscillator, which is the semigroup
\begin{displaymath}
  \rho_t = \exp(-t (\Delta + \|x\|^2)) \,,
\end{displaymath}
acting on $L^2(\R^n_x)$.  This state is the $n$-th tensor product of
the one-dimensional case. Furthermore it is self-dual in Fourier
transform, i.e. $\widehat \rho_t = \rho_t$. The kernel of $\rho_t$ is
given by Mehler's formula
(see~\cite[Chap.~4.2]{Berline-Getzler-Vergne}):
\begin{displaymath}
  \rho_t(x,y) = (2\pi \sinh 2t)^{-n/2} \exp \bigl( -\frac{1}{2} ( \coth
  2t) (\|x\|^2 + \|y\|^2) + (\sinh 2t)^{-1} \langle x,y \rangle \bigr) \,,
\end{displaymath}
so that its density and trace are
\begin{displaymath}
  \frac{d \nu_{\rho_t}}{dx} = \rho_t(x,x) = (2\pi \sinh 2t)^{-n/2}
  \exp(- (\tanh t) \|x\|^2) \quad \mathrm{and} \quad \tau(\rho_t) =
  (2 \sinh t)^{-n}\,.
\end{displaymath}
This leads easily to the entropies of the normalized states $
\lambda_t = \rho_t / \tau(\rho_t)$, namely
\begin{displaymath}
  S_x(\lambda_t) = \frac{n}{2} - \frac{n}{2} \ln (\frac{\tanh
    t}{\pi}) \quad \mathrm{while}\quad S_\xi (\widehat{ \lambda_t}) =
  \frac{n}{2} - \frac{n}{2} \ln 
  (\frac{\tanh t}{\pi}) - n \ln (2\pi)  \,,
\end{displaymath}
hence
\begin{displaymath}
  S_x (\lambda_t) + S_\xi(\widehat\lambda_t) = n -n \ln 2 - n\ln(\tanh t)\,.
\end{displaymath}
To compute von Neumann's entropy of $\lambda_t$, we recall that on
$\R$, the spectrum of $\rho_t$ is given by $p_k = e^{-(2k+1)t}$, $k
\in\N $. One finds that
\begin{align*}
  S(\lambda_t) & = n S(\lambda_t^\R) = - n \sum_{k\geq 0} (2\sinh
  t)p_k \ln
  ((2\sinh t)p_k) \\
  & = - n \ln (2\sinh t) - 2 n t \sinh t \sum_{k\geq 0} (-2k-1)
  e^{-(2k+1)t}\\
  & = - n \ln (2\sinh t) - 2nt\sinh t \bigl(\frac{1}{2\sinh t}\bigr)'
  \\
  & = nt \coth t - n \ln 2 -n \ln (\sinh t)\,.
\end{align*}
Therefore we obtain
\begin{displaymath}
  S_x (\lambda_t) + S_\xi(\widehat\lambda_t) - S(\lambda_t)  = n(1 + \ln
  (\cosh t) - t \coth t) \,,
\end{displaymath}
which is easily seen to be increasing in $t$ and positive. Hence these
states also satisfy the entropy bound \eqref{eq:54}, sharply when $t$
goes to $0$.

% \bibliographystyle{alpha} \bibliographystyle{amsplain}
%\bibliographystyle{abbrv}
% \bibliographystyle{smfplain}

%\bibliography{sobolev}

\def\cprime{$'$}

-----------------------------------------------------

\end{document}